\documentclass[12pt]{amsart}
\usepackage{amssymb}
\usepackage{amsfonts}
\usepackage{amssymb,latexsym}
\usepackage{enumerate}
\usepackage{mathrsfs}
\usepackage{hyperref}
 \usepackage{color}\makeatletter
\@namedef{subjclassname@2010}{%
  \textup{2010} Mathematics Subject Classification}
\makeatother

  [2002/01/19 v2.2g %
    AMS font definitions%
  ]
\DeclareFontFamily{U}{euf}{}
\DeclareFontShape{U}{euf}{m}{n}{%
  <5><6><7><8><9>gen*eufm%
  <10><10.95><12><14.4><17.28><20.74><24.88>eufm10%
  }{}
\DeclareFontShape{U}{euf}{b}{n}{%
  <5><6><7><8><9>gen*eufb%
  <10><10.95><12><14.4><17.28><20.74><24.88>eufb10%
  }{}

  [2002/01/19 v2.2g %
    AMS font definitions%
  ]
\DeclareFontFamily{U}{msb}{}
\DeclareFontShape{U}{msb}{m}{n}{%
  <5><6><7><8><9>gen*msbm%
  <10><10.95><12><14.4><17.28><20.74><24.88>msbm10%
  }{}

  [2002/01/19 v2.2g %
    AMS font definitions%
  ]
\DeclareFontFamily{U}{msa}{}
\DeclareFontShape{U}{msa}{m}{n}{%
  <5><6><7><8><9>gen*msam%
  <10><10.95><12><14.4><17.28><20.74><24.88>msam10%
  }{}

\newtheorem{theorem}{Theorem}[section]
\newtheorem{lemma}[theorem]{Lemma}
\newtheorem{proposition}[theorem]{Proposition}
\newtheorem{corollary}[theorem]{Corollary}

\theoremstyle{definition}
\newtheorem{definition}[theorem]{Definition}

\newtheorem{remark}[theorem]{Remark}

\numberwithin{equation}{section} 

\begin{document}

\title[Infinite order linear differential equation]
{Infinite order linear differential equation satisfied  by $p$-adic Hurwitz-type Euler zeta functions}

\dedicatory{Dedicated to the memory of Prof. David Goss (1952-2017)}


\author{Su Hu}
\address{Department of Mathematics, South China University of Technology, Guangzhou, Guangdong 510640, China}
\email{mahusu@scut.edu.cn}

\author{Min-Soo Kim}
\address{Department of Mathematics Education, Kyungnam University, Changwon, Gyeongnam 51767, Republic of Korea}
\email{mskim@kyungnam.ac.kr}




\begin{abstract}
In 1900, at the international congress of mathematicians, Hilbert  claimed that the Riemann zeta function $\zeta(s)$
is not the solution of any algebraic ordinary differential equations on its region of analyticity. In 2015, Van Gorder  \cite{VanGorder}
considered the question of whether $\zeta(s)$ satisfies a non-algebraic differential equation and  showed that it \emph{formally}
satisfies an infinite order linear differential equation. Recently, Prado and Klinger-Logan \cite{PK}
extended Van Gorder's result to show that the Hurwitz zeta function $\zeta(s,a)$ is also \emph{formally} satisfies a similar differential
equation
\begin{equation*}\label{HurDE}
    T\left[\zeta (s,a) - \frac{1}{a^s}\right] = \frac{1}{(s-1)a^{s-1}}.
\end{equation*}
But unfortunately in the same paper they proved that  the operator $T$ applied to Hurwitz zeta function $\zeta(s,a)$
does not converge at any point in the complex plane $\mathbb{C}$.

In this paper,  by defining $T_{p}^{a}$, a $p$-adic analogue of Van Gorder's operator $T,$  we establish an analogue of Prado and Klinger-Logan's
differential equation satisfied by $\zeta_{p,E}(s,a)$
which is the $p$-adic analogue of the Hurwitz-type Euler zeta functions
\begin{equation*}\label{HEZ}
\zeta_E(s,a)=\sum_{n=0}^\infty\frac{(-1)^n}{(n+a)^s}.
\end{equation*}
In contrast with the complex case, due to the non-archimedean property, the operator $T_{p}^{a}$ applied to the $p$-adic Hurwitz-type Euler zeta function $\zeta_{p,E}(s,a)$ is convergent
 $p$-adically in the area of $s\in\mathbb{Z}_{p}$ with $s\neq 1$ and $a\in K$ with $|a|_{p}>1,$ where $K$ is any finite extension of
 $\mathbb{Q}_{p}$ with ramification index over $\mathbb{Q}_{p}$  less than $p-1.$\end{abstract}

\maketitle

\section{Introduction}

Throughout this paper we shall use the following notations.
\begin{equation*}
\begin{aligned}
\qquad \mathbb{C}  ~~~&- ~\textrm{the field of complex numbers}.\\
\qquad p  ~~~&- ~\textrm{an odd rational prime number}. \\
\qquad\mathbb{Z}_p  ~~~&- ~\textrm{the ring of $p$-adic integers}. \\
\qquad\mathbb{Q}_p~~~&- ~\textrm{the field of fractions of}~\mathbb Z_p.\\
\qquad\mathbb C_p ~~~&- ~\textrm{the completion of a fixed algebraic closure}~\overline{\mathbb Q}_p~ \textrm{of}~\mathbb Q_{p}.
\end{aligned}
\end{equation*}

The Riemann zeta function $\zeta(s)$ is defined as
\begin{equation}\label{Riemann}
\zeta(s)=\sum_{n=1}^{\infty}\frac{1}{n^{s}},\quad\textrm{Re}(s) > 1,
\end{equation}
it can be analytically continued to the whole complex plane except for a single pole at $s=1$ with residue 1.
In 1900, at the international congress of mathematicians, David Hilbert \cite{HilbertProb} claimed that  $\zeta(s)$
is not the solution of any algebraic ordinary differential equations on its region of analyticity.
In 2015, Van Gorder \cite{VanGorder} considered the question of whether $\zeta(s)$ satisfies a non-algebraic differential equation and  showed that it \emph{formally} satisfies an infinite order linear differential equation.
In fact, he established the differential equation
\begin{equation}\label{zetaDE}
T[\zeta(s)-1]=\frac{1}{s-1}
\end{equation}
formally, where
\begin{equation}\label{T}
    T = \sum_{n=0}^\infty L_n
\end{equation}
and
\begin{align*}
    L_n &:= p_n(s) \exp(nD), \\
    p_n(s) &:= \begin{cases}
    1& \text{ if } n = 0 \\
    \frac{1}{(n+1)!}\prod_{j = 0}^{n-1} (s+j) &\text{ if } n >0,
    \end{cases}\\
    \exp(nD) &:= id + \sum_{k=1}^\infty \frac{n^k}{k!} D_s^k
\end{align*}
for $D_s^k := \frac{\partial ^k}{\partial s^k}.$

For $0 < a \leq 1$, Re$(s) >1$, in 1882 Hurwitz~\cite{Hurwitz}
defined the partial zeta functions
\begin{equation}~\label{Hurwitz}
\zeta(s,a)=\sum_{n=0}^{\infty}\frac{1}{(n+a)^{s}},
\end{equation}
which generalized (\ref{Riemann}).
As (\ref{Riemann}), this function can also be analytically continued to a meromorphic function in the
complex plane with a simple pole at $s=1$.
Recently, Prado and Klinger-Logan \cite{PK}
extended Van Gorder's result to show that the Hurwitz zeta function $\zeta(s,a)$ also \emph{formally} satisfies a similar differential
equation
\begin{equation}\label{HurDE1}
    T\left[\zeta (s,a) - \frac{1}{a^s}\right] = \frac{1}{(s-1)a^{s-1}}
\end{equation}
for $s \in \mathbb{C}$ satisfying $s + n \neq 1$ for all $n \in \mathbb{Z}_{\geq 0},$ where $T$ is the Van Gorder's operator
defined as in (\ref{T}) (see \cite[Corollary 4]{PK}). But unfortunately, in the same paper they proved that
$$T\left[ \zeta (s, a) - \frac{1}{a^s} \right]=\sum_{n=0}^{\infty} p_n(s)\exp(nD)\left[\zeta(s,a)- \frac{1}{a^s}\right],$$
the operator $T$ applied to Hurwitz zeta function, does not converge at any point in the complex plane $\mathbb{C}$ (see \cite[Theorem 8]{PK}).
Then they defined a generalized operator $G$ instead of $T$. That is, let $\mathcal{M}$ be the collection of meromorphic functions on $\mathbb{C}$ and $f\in\mathcal{M}$, define $G:\mathcal{M}\to \mathcal{M} $ by
\begin{align}\label{G}
G[f](s) = \sum_{n=0}^\infty p_n(s)f(s+n).
\end{align}
Under this linear operator, we have a convergent difference equation
\begin{equation}\label{HurDE2}
    G\left[\zeta (s,a) - \frac{1}{a^s}\right] = \frac{1}{(s-1)a^{s-1}}.
\end{equation}
But it needs to mention that $G$ is not a differential operator.

For Re$(s)>0$, the Euler zeta function (also called alternative series or Dirichlet  eta function)
is defined by
\begin{equation}~\label{Euler}
\zeta_{E}(s)=\sum_{n=1}^{\infty}\frac{(-1)^{n-1}}{n^{s}}.
\end{equation}
This function can be analytically continued  to the complex plane without any pole. For Re$(s)>0$, (\ref{Riemann}) and (\ref{Euler}) are connected by the following equation
\begin{equation}~\label{Riemann-Euler}
\zeta_{E}(s)=(1-2^{1-s})\zeta(s).
\end{equation}
By Weil's history~\cite[p.~273--276]{Weil} (also see a survey by Goss~\cite[Section 2]{Goss}),
Euler used (\ref{Euler}) to ``prove"
\begin{equation}~\label{fe}
\frac{\zeta_{E}(1-s)}{\zeta_{E}(s)}=\frac{-\Gamma(s)(2^{s}-1)\textrm{cos}(\pi s/2)}{(2^{s-1}-1)\pi^{s}},
\end{equation}
which leads to the functional equation of $\zeta(s)$.

For $s\in\mathbb{C}$ and $a\neq0,-1,-2,\ldots,$ the Hurwitz-type Euler zeta function is defined as the Hurwitz zeta function (\ref{Hurwitz}) twisted by $(-1)^{n}$
\begin{equation}\label{HEZ1}
\zeta_E(s,a)=\sum_{n=0}^\infty\frac{(-1)^n}{(n+a)^s}.
\end{equation}
This  function can  also be analytically
continued to the complex plane without any pole.
It represents a partial zeta function of cyclotomic fields in one version of Stark's conjectures in algebraic number theory (see \cite[p. 4249, (6.13)]{HK-G}).
Recently, several interesting properties for the function $\zeta_{E}(s,a)$ have been studied, including its Fourier expansion and several integral representations \cite{HKK}, special values and power series expansions \cite{HK2019},
convexity properties \cite{Cvijovic}, etc.

In \cite{KS}, using the fermionic $p$-adic integral (see (\ref{-q-e2}) below),
we defined $\zeta_{p,E}(s,a),$ the $p$-adic analogue of Hurwitz-type Euler zeta functions (\ref{HEZ1}), which interpolates (\ref{HEZ1}) at nonpositive integers (see Theorem \ref{p-E-zeta-val} below),
so called the $p$-adic Hurwitz-type Euler zeta functions.
In the same paper, we also proved several properties of $\zeta_{p,E}(s,a),$
including the analyticity, the convergent Laurent series
expansion, the distribution formula, the difference equation,
the reflection functional equation, the derivative formula and the $p$-adic Raabe formula.

In this note, we define a $p$-adic analogue of the operator $T,$ denoted by $T_{p}^{a}$ (see (\ref{Tp}) below). Under this operator, the $p$-adic Hurwitz-type Euler zeta function $\zeta_{p,E}(s,a)$ satisfies
an infinite order linear differential equation
\begin{equation}\label{main}
T_{p}^{a}\left[\zeta_{p,E}(s,a)-\langle a\rangle^{1-s}\right]
=\frac{1}{s-1}\left(\langle a-1 \rangle^{1-s}-\langle a\rangle^{1-s}\right)
\end{equation}
(see Theorem \ref{main theorem}).
In contrast with the complex case, the left hand side of the above equation is convergent everywhere for $s\in\mathbb{Z}_{p}$ with $s\neq 1$ and $a\in K$ with $|a|_{p}>1,$ where $K$ is any finite extension of
 $\mathbb{Q}_{p}$ with ramification index  over $\mathbb{Q}_{p}$  less than $p-1$ (see Corollary~\ref{Cor} and Remarks~\ref{remark3} and \ref{remark4} below).

\section{Preliminaries}

\subsection{$p$-adic Teichm\"uller character}\label{2.1}

To our purpose, in this subsection, we recall some notions from $p$-adic analysis,
including the $p$-adic Teichm\"uller character $\omega_v(a)$ and
the projection function $\langle a\rangle$ for $a\in\mathbb{C}_{p}^{\times}$.
Our approach follows Tangedal and Young in \cite{TP} closely.

Given $a\in\mathbb{Z}_{p}, p\nmid a$ and $p>2,$ there exists a unique $(p-1)$th
root of unity $\omega(a)\in\mathbb Z_p$ such that
$$a\equiv\omega(a)\pmod{p},$$
where $\omega$ is the Teichm\"uller character. Let $\langle
a\rangle=\omega^{-1}(a)a,$ so $\langle a\rangle \equiv1\pmod p.$

In what follows we extend the definition domain of the projection  function $\langle a\rangle$ from $\mathbb{Z}_{p}$ to $\mathbb{C}_{p}$. Fixed an embedding of $\overline{\mathbb{Q}}$ into $\mathbb{C}_{p},$
denote the image of the set of positive real rational powers of $p$ under this
embedding  in $\mathbb{C}_{p}^{\times}$  by $p^{\mathbb{Q}},$
and  the group of roots of unity with order not divisible by $p$  in
$\mathbb{C}_{p}^{\times}$ by $\mu$. Given
$a\in\mathbb{C}_{p}$ with $|a|_{p}=1$, there exists a unique element
$\hat{a}\in\mu$ such that
 \begin{equation}\label{hat+}|a-\hat{a}|_{p} < 1, \end{equation}
which is also named the
Teichm\"uller representative of $a$; it can also be defined from
$\hat{a}=\lim_{n\to\infty}a^{p^{n!}}$. Then we extend this definition to
$a\in\mathbb{C}_{p}^{\times}$ by
\begin{equation}\label{hat}
\hat{a}:=(\widehat{a/p^{v_{p}(a)}}),
\end{equation}
that is, we define $\hat{a}=\hat{u}$ if $a=p^{r}u$ with $p^{r}\in
p^{\mathbb{Q}}$ and $|u|_{p}=1$, then we define the function
$\langle\cdot\rangle$ on $\mathbb{C}_{p}^{\times}$ by
$$\langle a\rangle=p^{-v_{p}(a)}a/\hat{a}.$$
Now we define $\omega_v(\cdot)$ on $\mathbb{C}_{p}^{\times}$ by
\begin{equation}\label{omega}
\omega_v(a)=\frac{a}{\langle a\rangle}=p^{v_{p}(a)}\hat{a}.
\end{equation}
From this we get an
internal product decomposition of multiplicative groups
\begin{equation}
\mathbb{C}_{p}^{\times}\simeq p^{\mathbb{Q}}\times\mu\times D,
\end{equation}
where $D=\{a\in\mathbb{C}_{p}: |a-1|_{p} < 1\},$ given by
\begin{equation}\label{deco}
a=p^{v_{p}(a)}\cdot\hat{a}\cdot\langle a\rangle\mapsto
(p^{v_{p}(a)},\hat{a},\langle a\rangle).
\end{equation}
As remarked by Tangedal and Young in \cite{TP}, this decomposition of
$\mathbb{C}_{p}^{\times}$ depends on the choice of embedding of
$\overline{\mathbb{Q}}$ into $\mathbb{C}_{p}$; the projections
$p^{v_{p}(a)},\hat{a},\langle a\rangle$ are uniquely determined up
to roots of unity. However for $a\in\mathbb{Q}_{p}^{\times}$ the
projections $p^{v_{p}(a)},\hat{a},\langle a\rangle$ are uniquely
determined and do not depend on the choice of the embedding. Notice
that the projections $a\mapsto p^{v_{p}(a)}$ and $a\mapsto \hat{a}$
are constant on discs of the form $\{a\in\mathbb{C}_{p}:|a-y|_{p} <
|y|_{p}\}$ and therefore have derivative zero whereas the
projections $a\mapsto\langle a\rangle$ has derivative
$\frac{d}{da}\langle a\rangle=\langle a\rangle/a$.

\subsection{The fermionic $p$-adic integral and the $p$-adic Hurwitz-type Euler zeta functions}

In this subsection, we recall the definition of the $p$-adic Hurwitz-type Euler zeta functions $\zeta_{p,E}(s,a)$ from the fermionic $p$-adic integral. For details, we refer to \cite{KS}.

 Let $UD(\mathbb Z_p)$ be the space of all uniformly (or strictly) differentiable $\mathbb C_p$-valued functions on $\mathbb Z_p$ (see \cite[\S11.1.2]{Cohen}).
The fermionic $p$-adic integral $I_{-1}(f)$  on $\mathbb Z_p$ of a function $f\in UD(\mathbb Z_p)$ is
defined by
\begin{equation}\label{-q-e2}
I_{-1}(f)=\int_{\mathbb Z_p}f(t)d\mu_{-1}(t)
=\lim_{r\rightarrow\infty}\sum_{k=0}^{p^r-1}f(k)(-1)^k.
\end{equation}
The fermionic $p$-adic integral (\ref{-q-e2}) was independently found by Katz \cite[p.~486]{Katz} (in Katz's notation, the $\mu^{(2)}$-measure), Shiratani and
Yamamoto \cite{Shi}, Osipov \cite{Osipov}, Lang~\cite{Lang} (in Lang's notation, the $E_{1,2}$-measure), T. Kim~\cite{TK} from very different viewpoints.

For $a\in\mathbb{C}_{p}^{\times}$ and $s\in\mathbb{C}_{p},$ the two-variable function
$\langle a\rangle^{s}$ (\cite[p. 141]{SC}) is defined by
\begin{equation}\label{integral}
\langle a\rangle^{s}=\sum_{n=0}^{\infty}\binom{s}{n}(\langle a\rangle-1)^{n},
\end{equation}
when this sum is convergence. The analytic property of $\langle a\rangle^{s}$
is stated in the following proposition.

\begin{proposition}[see Tangedal and Young \cite{TP}]\label{analytic}
For any
 $a\in\mathbb{C}_{p}^{\times}$ the function $s\mapsto\langle
 a\rangle^{s}$ is a $C^{\infty}$ function of $s$ on
 $\mathbb{Z}_{p}$ and is analytic on a disc of positive radius about
 $s=0$; on this disc it is locally analytic as a function of $a$ and
 independent of the choice made to define the $\langle
 \cdot\rangle$ function. If $a$ lies in a finite extension $K$ of
 $\mathbb{Q}_{p}$ whose ramification index over $\mathbb{Q}_{p}$ is
 less than $p-1$ then  $s\mapsto\langle
 a\rangle^{s}$  is analytic for $|s|_{p} <
 |\pi|_{p}^{-1}p^{-1/(p-1)}$, where $(\pi)$ is the maximal ideal of
 the ring of integers $O_{K}$ of $K$. If $s\in\mathbb{Z}_{p},$
 the function $a\mapsto\langle a\rangle^{s}$ is an analytic function of $a$
 on any disc of the form $\{a\in\mathbb{C}_{p}:|a-y|_{p} <|y|_{p}\}$.
\end{proposition}

Now we are at the position to recall the definition for the $p$-adic Hurwitz-type Euler zeta function  $\zeta_{p,E}(s,a)$.

\begin{definition}[{see \cite[Definition 3.3]{KS}}]\label{p-E-zeta}
For $a\in \mathbb{C}_{p}\backslash \mathbb{Z}_{p}$, we define the
$p$-adic Hurwitz-type Euler zeta function $\zeta_{p,E}(s,a)$ by the formula
\begin{equation}\label{E-zeta}
\zeta_{p,E}(s,a)=\int_{\mathbb{Z}_{p}}\langle a+t\rangle^{1-s}d\mu_{-1}(t).
\end{equation}
\end{definition}

The following theorem summarize the analytic property of $\zeta_{p,E}(s,x)$ and
Tangedal and Young proved a similar result for
$p$-adic multiple zeta functions (see \cite[Theorem 3.1]{TP}).

\begin{theorem}[{see \cite[Theorem 3.4]{KS}}] \label{analytic-zeta}
For any  choice of $a\in\mathbb{C}_{p}\backslash \mathbb{Z}_{p}$ the
function $\zeta_{p,E}(s,a)$ is a $C^{\infty}$ function of $s$ on
$\mathbb{Z}_{p}$, and is an analytic function of $s$ on a disc of
positive radius about $s=0$; on this disc it is locally analytic as
a function of $a$ and independent of the choice made to define the
$\langle\cdot\rangle$ function. If $a$ is so chosen to lie in a
finite extension $K$ of $\mathbb{Q}_{p}$ whose ramification index
over $\mathbb{Q}_{p}$ is less than $p-1$ then $\zeta_{p,E}(s,a)$ is
analytic for $|s|_{p} <
 |\pi|_{p}^{-1}p^{-1/(p-1)}$. If $s\in\mathbb{Z}_{p}$, the function
 $\zeta_{p,E}(s,a)$ is locally analytic as a function of $a$ on
 $\mathbb{C}_{p}\backslash\mathbb{Z}_{p}$.
\end{theorem}

It needs to mention that the $p$-adic Hurwitz-type Euler zeta function $\zeta_{p,E}(s,a)$ interpolates its complex counterpart $\zeta_{E}(s,a)$ (\ref{HEZ1}) $p$-adically, that is,

\begin{theorem}[{see \cite[Theorem 3.8]{KS}}] \label{p-E-zeta-val}
Suppose that $a\in \mathbb{C}_{p}$ and $|a|_p>1.$
For $m\in\mathbb N,$
$$\zeta_{p,E}(1-m,a)=\frac1{\omega_v^{m}(a)}E_m(a)=\frac1{\omega_v^{m}(a)}\zeta_E(-m,a),$$
where the Euler  polynomials $E_{m}(x)$ is defined by the generating function
\begin{equation}\label{Eu-pol}
\frac{2e^{xz}}{e^z+1}=\sum_{m=0}^\infty E_m(x)\frac{z^m}{n!}, \quad |z|<\pi.
\end{equation}
\end{theorem}

\subsection{The $p$-adic operator $T_{p}^{a}$}

In this subsection,  we give a definition of $T_{p}^{a}$, the $p$-adic analogue of the operator $T$ (see (\ref{T})).
 Let $E=\{x\in\mathbb{C}_{p}: |x|_{p}<p^{-\frac{1}{p-1}}\}$ be the region of  convergence of the power series
$\sum_{k=0}^{\infty}\frac{x^{k}}{k!}$. The $p$-adic exponential function is given by
$$\textrm{exp}_{p}(x)=\sum_{k=0}^{\infty}\frac{x^{k}}{k!},\quad(x\in E)$$
(see \cite[p. 70]{SC}) and the $p$-adic Van Gorder's operator is defined as follows
\begin{equation}\label{Tp}
    T _{p}^{a}= \sum_{n=0}^\infty L_{p,n}^{a},
\end{equation}
where
\begin{equation}\label{Tp2}
\begin{aligned}
    L_{p,n}^{a}&:= P_{p,n}^{a}(s) \exp_{p}(nD), \\
    P_{p,n}^{a}(s) &:= \begin{cases}
     \frac{2}{s-1}& \text{ if } n = 0\\
        \frac{1}{\omega_{v}(a)}& \text{ if } n = 1\\
    \frac{1}{n!\omega_{v}^{n}(a)}\prod_{j = 1}^{n-1} (s-1+j) &\text{ if } n\geq 2,
    \end{cases}  \\
    \exp_{p}(nD) &:= id + \sum_{k=1}^\infty \frac{n^k}{k!} D_s^k
\end{aligned}
\end{equation}
for $D_s^k := \frac{\partial ^k}{\partial s^k}.$

\section{Main results}

In this section, we shall prove (\ref{main}). First we need to establish the following identity for $\zeta_{p,E}(s,a)$.

\begin{lemma}\label{Lemma1}
Let $\zeta_{p,E}(s,a)$ be the $p$-adic Hurwitz-type Euler zeta function. Then, for $s\in\mathbb{Z}_{p}$ with $s\neq 1$ and $a\in\mathbb{C}_{p}$ with $|a|_{p}>1$,
we have that
  \begin{equation}\label{main1}
\begin{aligned}
&\quad \frac{2}{s-1}\left(\zeta_{p,E}(s,a)-\langle a \rangle^{1-s}\right)
+\frac{1}{\omega_{v}(a)}\left(\zeta_{p,E}(s+1,a)-\langle a \rangle^{1-(s+1)}\right)\\
&\quad+\sum_{n=2}^\infty \frac{\prod_{j = 1}^{n - 1}(s-1+j)}{n!\omega_{v}^{n}(a)}\left(\zeta_{p,E}(s+n,a)-\langle a \rangle^{1-(s+n)}\right)\\
&=\frac{1}{s-1}\left(\langle a-1 \rangle^{1-s}-\langle a \rangle^{1-s}\right).
\end{aligned}
\end{equation}
\end{lemma}

\begin{remark}
This is a $p$-adic analogue of complex identities  for the Hurwitz zeta function $\zeta(s,a)$ (see \cite[Lemma 1]{PK}) and  for the Riemann zeta function $\zeta(s)$ (see \cite[(3.3)]{PK}).
\end{remark}

\begin{proof}[Proof of Lemma \ref{Lemma1}.]
Fix $s\in\mathbb{Z}_{p}$ and $a\in\mathbb{C}_{p}$ with $|a|_{p}>1.$ For any $r\in\mathbb{N}$ we have
\begin{equation}\label{1}
\begin{aligned}
-\frac{1}{(s-1)\langle a\rangle^{s-1}}
&=\frac{1}{s-1}\left(\sum_{k=0}^{p^{r}-1}\frac{(-1)^{k+1}}{\langle k+a\rangle^{s-1}}+\sum_{k=1}^{p^{r}-1}\frac{(-1)^{k}}{\langle k+a\rangle^{s-1}}\right)\\
&=\frac{1}{s-1}\left(\sum_{k=1}^{p^{r}}\frac{(-1)^{k}}{\langle k-1+a\rangle^{s-1}}+\sum_{k=1}^{p^{r}-1}\frac{(-1)^{k}}{\langle k+a\rangle^{s-1}}\right)\\
&=\frac{1}{s-1}\left(\sum_{k=1}^{p^{r}-1}\frac{(-1)^{k}}{\langle k-1+a\rangle^{s-1}}+\sum_{k=1}^{p^{r}-1}\frac{(-1)^{k}}{\langle k+a\rangle^{s-1}}+\frac{(-1)^{p^{r}}}{\langle p^{r}-1+a\rangle^{s-1}}\right)\\
&=\frac{1}{s-1}\sum_{k=1}^{p^{r}-1}\frac{(-1)^{k}}{\langle k+a \rangle^{s-1}}\left(\left\langle \frac{k+a}{k-1+a}\right\rangle^{s-1}+1\right) \\
&\quad-\frac{1}{s-1}\frac{1}{\langle p^{r}-1+a\rangle^{s-1}}
\\&\quad\text{(since $p$ is an odd prime)} .
\end{aligned}
\end{equation}
Since $|a|_{p}>1$, for $k\in\mathbb{N}$, we have $|k+a|_{p}>1$, thus 
$$\left|1-\frac{1}{k+a}\right|_{p}=1$$ 
and  
$$\left|\frac{1}{1-\frac{1}{k+a}}-1\right|_{p}=\left|\frac{\frac{1}{k+a}}{1-\frac{1}{k+a}}\right|_{p}<1.$$
Then from (\ref{hat+}) we see that $$\widehat{\frac{1}{1-\frac{1}{k+a}}}=1$$ and  by (\ref{omega})  $$\omega_{v}\left( \frac{1}{1-\frac{1}{k+a}}\right)=1.$$
Again by (\ref{omega}), we have
\begin{equation}\label{2}
\begin{aligned}
\left\langle \frac{k+a}{k-1+a}\right\rangle&=\left\langle\frac{1}{1-\frac{1}{k+a}}\right\rangle\\
&=\omega_{v}^{-1}\left(\frac{1}{1-\frac{1}{k+a}}\right)\left(\frac{1}{1-\frac{1}{k+a}}\right)\\
&=\frac{1}{1-\frac{1}{k+a}}\\
&=\left(1-\frac{1}{k+a}\right)^{-1}.
\end{aligned}
\end{equation}
From \cite[p.140, Lemma 47.6]{SC}, for $s\in\mathbb{Z}_{p}$ we have the expansion
$$(1+x)^{s}=\sum_{n=0}^{\infty}\binom{s}{n}x^{n},\quad|x|_{p}<1.$$
Thus by (\ref{2}) we get
\begin{equation}\label{3}
\begin{aligned}
\left\langle \frac{k+a}{k-1+a}\right\rangle^{s-1}&=\left(1-\frac{1}{k+a}\right)^{1-s}=\sum_{n=0}^{\infty}\binom{1-s}{n}\frac{(-1)^{n}}{(k+a)^{n}}\\
&=1+\sum_{n=1}^{\infty} \frac{\prod_{j = 0}^{n - 1}(s-1+j)}{n!}\frac{1}{(k+a)^{n}}.
\end{aligned}
\end{equation}
Substituting the above expansion into (\ref{1}), we have
\begin{equation}\label{4}
\begin{aligned}
&\quad-\frac{1}{(s-1)\langle a\rangle^{s-1}}\\
&=\frac{1}{s-1}\sum_{k=1}^{p^{r}-1}\frac{(-1)^{k}}{\langle k+a \rangle^{s-1}}\left(\left(1+\sum_{n=1}^{\infty}\frac{\prod_{j = 0}^{n - 1}(s-1+j)}{n!}\frac{1}{(k+a)^{n}}\right)+1\right)\\
&\quad-\frac{1}{s-1}\frac{1}{\langle p^{r}-1+a\rangle^{s-1}}\\
&=\frac{2}{s-1}\sum_{k=1}^{p^{r}-1}\frac{(-1)^{k}}{\langle k+a \rangle^{s-1}} \\
&\quad+\frac{1}{s-1}\sum_{k=1}^{p^{r}-1}\frac{(-1)^{k}}{\langle k+a \rangle^{s-1}}\sum_{n=1}^\infty \frac{\prod_{j = 0}^{n - 1}(s-1+j)}{n!}\frac{1}{(k+a)^n}
\\
&\quad-\frac{1}{s-1}\frac{1}{\langle p^{r}-1+a\rangle^{s-1}}\\
&=\frac{2}{s-1}\sum_{k=1}^{p^{r}-1}\frac{(-1)^{k}}{\langle k+a \rangle^{s-1}} \\
&\quad+\frac{1}{s-1}\sum_{k=1}^{p^{r}-1}\frac{(-1)^{k}}{\langle k+a \rangle^{s-1}}\left(\frac{s-1}{k+a}+\sum_{n=2}^\infty \frac{\prod_{j = 0}^{n - 1}(s-1+j)}{n!}\frac{1}{(k+a)^n}\right)
\\
&\quad-\frac{1}{s-1}\frac{1}{\langle p^{r}-1+a\rangle^{s-1}}.
\end{aligned}
\end{equation}
Since $|a|_{p}>1$ and $k\in\mathbb{N}$, by (\ref{omega}) we have 
$$\omega_{v}(k+a)=\omega_{v}(a)$$ 
and
$$k+a=\omega_{v}(k+a)\langle k+a \rangle=\omega_{v}(a)\langle k+a \rangle.$$
Substituting the above identity into (\ref{4}), we get
\begin{equation}\label{4+}
\begin{aligned}
-\frac{1}{(s-1)\langle a\rangle^{s-1}}
&=\frac{2}{s-1}\sum_{k=1}^{p^{r}-1}\frac{(-1)^{k}}{\langle k+a \rangle^{s-1}}+\frac{1}{\omega_{v}(a)}\sum_{k=1}^{p^{r}-1}\frac{(-1)^{k}}{\langle k+a \rangle^{s}}\\&\quad+\sum_{n=2}^\infty \frac{\prod_{j = 1}^{n - 1}(s-1+j)}{n!\omega_{v}^{n}(a)}\sum_{k=1}^{p^{r}-1}\frac{(-1)^{k}}{\langle k+a \rangle^{s+n-1}} \\
&\quad-\frac{1}{s-1}\frac{1}{\langle p^{r}-1+a\rangle^{s-1}}.\\\end{aligned}
\end{equation}
Taking $r\to\infty$ in the above equality, by the continuity of the
$p$-adic function $\langle a \rangle^{s}$ in $a$ (see the last sentence of Proposition \ref{analytic}), we have $$\lim_{r\to\infty}\langle p^{r}-1+a\rangle^{s-1} =\langle a-1\rangle^{s-1}$$ and
 \begin{equation}\label{5}
\begin{aligned}
-\frac{1}{(s-1)\langle a\rangle^{s-1}}
&=\frac{2}{s-1}\lim_{r\to\infty}\sum_{k=1}^{p^{r}-1}\frac{(-1)^{k}}{\langle k+a \rangle^{s-1}}+\frac{1}{\omega_{v}(a)}\lim_{r\to\infty}\sum_{k=1}^{p^{r}-1}\frac{(-1)^{k}}{\langle k+a \rangle^{s}}
\\
&\quad+\sum_{n=2}^\infty \frac{\prod_{j = 1}^{n - 1}(s-1+j)}{n!\omega_{v}^{n}(a)}\lim_{r\to\infty}\sum_{k=1}^{p^{r}-1}\frac{(-1)^{k}}{\langle k+a \rangle^{s+n-1}}\\
&\quad (\textrm{see Proposition \ref{proposition-add}})\\
&\quad-\frac{1}{s-1}\frac{1}{\langle a-1\rangle^{s-1}}\\
&=\frac{2}{s-1}\left(\lim_{r\to\infty}\sum_{k=0}^{p^{r}-1}\frac{(-1)^{k}}{\langle k+a \rangle^{s-1}}-\frac{1}{\langle a \rangle^{s-1}}\right)
\\
&\quad+\frac{1}{\omega_{v}(a)}\left(\lim_{r\to\infty}\sum_{k=0}^{p^{r}-1}\frac{(-1)^{k}}{\langle k+a \rangle^{s}}-\frac{1}{\langle a \rangle^{s}}\right)\\
&\quad+\sum_{n=2}^\infty \frac{\prod_{j = 1}^{n - 1}(s-1+j)}{n!\omega_{v}^{n}(a)}\left(\lim_{r\to\infty}\sum_{k=0}^{p^{r}-1}\frac{(-1)^{k}}{\langle k+a \rangle^{s+n-1}}-\frac{1}{\langle a \rangle^{s+n-1}}\right) \\
&\quad-\frac{1}{s-1}\frac{1}{\langle a-1\rangle^{s-1}}.
\end{aligned}
\end{equation}
Then by the definitions of the fermionic $p$-adic integral (\ref{integral}) and the $p$-adic Hurwitz-type zeta function $\zeta_{p,E}(s,a)$ (\ref{E-zeta}), we have
\begin{equation}\label{6}
\begin{aligned}
-\frac{1}{(s-1)\langle a\rangle^{s-1}}
&=\frac{2}{s-1}\left(\int_{\mathbb{Z}_{p}}\langle k+a \rangle^{1-s}d\mu_{-1}(a)-\frac{1}{\langle a \rangle^{s-1}}\right)\\
&\quad+\frac{1}{\omega_{v}(a)}\left(\int_{\mathbb{Z}_{p}}\langle k+a \rangle^{-s}d\mu_{-1}(a)-\frac{1}{\langle a \rangle^{s}}\right)\\
&\quad+\sum_{n=2}^\infty \frac{\prod_{j = 1}^{n - 1}(s-1+j)}{n!\omega_{v}^{n}(a)}\left(\int_{\mathbb{Z}_{p}}\langle k+a \rangle^{1-(s+n)}d\mu_{-1}(a)-\frac{1}{\langle a \rangle^{s+n-1}}\right) \\
&\quad-\frac{1}{s-1}\frac{1}{\langle a-1\rangle^{s-1}}\\
&=\frac{2}{s-1}\left(\zeta_{p,E}(s,a)-\langle a \rangle^{1-s}\right)
\\&\quad+\frac{1}{\omega_{v}(a)}\left(\zeta_{p,E}(s+1,a)-\langle a \rangle^{1-(s+1)}\right)\\
&\quad+\sum_{n=2}^\infty \frac{\prod_{j = 1}^{n - 1}(s-1+j)}{n!\omega_{v}^{n}(a)}\left(\zeta_{p,E}(s+n,a)-\langle a \rangle^{1-(s+n)}\right) \\
&\quad-\frac{1}{s-1}\frac{1}{\langle a-1\rangle^{s-1}},
\end{aligned}
\end{equation}
which is equivalent to
 \begin{equation}\label{6-1}
\begin{aligned}
&\quad \frac{2}{s-1}\left(\zeta_{p,E}(s,a)-\langle a \rangle^{1-s}\right)
+\frac{1}{\omega_{v}(a)}\left(\zeta_{p,E}(s+1,a)-\langle a \rangle^{1-(s+1)}\right)\\
&\quad+\sum_{n=2}^\infty \frac{\prod_{j = 1}^{n - 1}(s-1+j)}{n!\omega_{v}^{n}(a)}\left(\zeta_{p,E}(s+n,a)-\langle a \rangle^{1-(s+n)}\right)\\
&=\frac{1}{s-1}\left(\langle a-1 \rangle^{1-s}-\langle a \rangle^{1-s}\right).
\end{aligned}
\end{equation}
This completes the proof.
\end{proof}

As pointed out by the referee, in order to move the limit to the inside of the
summation $\sum_{n=2}^{\infty}$ in (\ref{5}) of the above lemma, we need to show that the convergence of the inner
limit is uniform  for $r\in\mathbb{N}$. To this end, we add the following proposition.

\begin{proposition}\label{proposition-add}
For $s\in\mathbb{Z}_{p}$ and $a\in\mathbb{C}_{p}$ with $|a|_{p}>1$, the series $$\sum_{n=2}^\infty \frac{\prod_{j = 1}^{n - 1}(s-1+j)}{n!\omega_{v}^{n}(a)}\sum_{k=1}^{p^{r}-1}\frac{(-1)^{k}}{\langle k+a \rangle^{s+n-1}}$$  converges uniformly for $r\in\mathbb{N}$ and 
\begin{equation}\begin{aligned}&\quad\lim_{r\to\infty}\sum_{n=2}^\infty \frac{\prod_{j = 1}^{n - 1}(s-1+j)}{n!\omega_{v}^{n}(a)}\sum_{k=1}^{p^{r}-1}\frac{(-1)^{k}}{\langle k+a \rangle^{s+n-1}}\\
&=\sum_{n=2}^\infty \frac{\prod_{j = 1}^{n - 1}(s-1+j)}{n!\omega_{v}^{n}(a)}\lim_{r\to\infty}\sum_{k=1}^{p^{r}-1}\frac{(-1)^{k}}{\langle k+a \rangle^{s+n-1}}.\end{aligned} \end{equation}\end{proposition}

\begin{proof}
For $n\geq 2$ we have
$$
 \frac{\prod_{j = 1}^{n - 1}(s-1+j)}{n!}=\binom{s+n-2}{n-1} \frac{1}{n}.
$$
By \cite[p. 138, Proposition 47.2(v)]{SC}, for $s\in\mathbb{Z}_{p},$
$$\left|\binom{s+n-2}{n-1}\right|_{p}\leq 1.$$
Since $|n|_{p}=\left(\frac{1}{p}\right)^{v_{p}(n)}$,
we have 
$$\left|\frac{1}{n}\right|_{p}=p^{v_{p}(n)}\leq n,$$
thus for $n\geq 2$,
\begin{equation}\label{l2}
\left|\frac{\prod_{j = 1}^{n - 1}(s-1+j)}{n!}\right|_{p}=\left|\binom{s+n-2}{n-1} \frac{1}{n}\right|_{p}\leq n.
\end{equation}
Since for any $a\in\mathbb{C}_{p}^{\times}$, $\hat{a}\in\mu$ is a root of unity, we have $|\hat{a}|_{p}=1$
and  by (\ref{omega})
$$|\omega_{v}(a)|_{p}=|p^{v_{p}(a)}\hat{a}|_{p}=|p|_{p}^{v_{p}(a)}=\left(\frac{1}{p}\right)^{v_{p}(a)},$$
thus
\begin{equation}\label{l3}
\left|\frac{1}{\omega_{v}^{n}(a)}\right|_{p}=p^{nv_{p}(a)}.
\end{equation}
Combining (\ref{l2}) and (\ref{l3}),  for any fixed $a\in\mathbb{C}_{p}$ with $|a|_{p}>1$ we have
\begin{equation}\label{est1}
\left|\frac{\prod_{j = 1}^{n - 1}(s-1+j)}{n!\omega_{v}^{n}(a)}
\right|_{p}\leq p^{nv_{p}(a)}\cdot n.
\end{equation}

Now fix $a\in\mathbb{C}_{p}$ with $|a|_{p}>1,$ we know that  $\log_{p}\langle y+a \rangle$ is a continuous function in $y\in\mathbb{Z}_{p}.$
Since $\mathbb{Z}_{p}$ is compact in the $p$-adic topology, by the Weierstrass maximum value theorem (\cite[p. 61, Theorem 4.3]{Lang2}) there exists a $y_{0}\in\mathbb{Z}_{p}$ such that 
\begin{equation} \label{s3} |\log_{p}\langle y+a \rangle|_{p} \leq |\log_{p}\langle y_{0}+a \rangle|_{p}\end{equation}
for all $y\in\mathbb{Z}_{p}$. 
Since $\langle y_{0}+a \rangle-1\in (p)$, we have
\begin{equation}\label{s5} |\langle y_{0}+a \rangle-1|_{p} \leq p^{-1} < p^{-1/(p-1)}\end{equation}
and by \cite[p. 51, Lemma 5.5]{Wa},  \begin{equation}\label{s6} |\log_{p}\langle y_{0}+a \rangle|_{p}=|\langle y_{0}+a \rangle-1|_{p}\leq p^{-1}.\end{equation}
Combining  (\ref{s3}), (\ref{s5}) and (\ref{s6}), we see that 
\begin{equation}\label{s1}
 |\log_{p}\langle y+a \rangle|_{p}\leq p^{-1}
 \end{equation}
 for all $y\in\mathbb{Z}_{p}$.
 By \cite[p. 1245, (2.22)]{TP}, for $(x,s)\in \mathbb{C}_{p}^{\times}\times\mathbb{C}_{p}$ satisfying $|s|_{p} < p^{-1/(p-1)}|\log_{p}\langle x \rangle|_{p}^{-1},$
 we have 
 \begin{equation}\label{s2}
 \langle x \rangle^{s}=\exp_{p}(s\log_{p} \langle x \rangle).
 \end{equation}
Let $D=\mathbb{Z}_{p}\times\mathbb{Z}_{p}$.  For $(y,s)\in D$, at first we have $|s|_{p}\leq 1$ and
by (\ref{s1}), we see that $p^{-1/(p-1)} |\log_{p}\langle y+a \rangle|_{p}^{-1}\geq p^{-1/(p-1)}\cdot p=p^{\frac{p-2}{p-1}}>1,$
thus $|s|_{p} <p^{-1/(p-1)}|\log_{p}\langle y+a \rangle|_{p}^{-1}.$
Then by (\ref{s2}) we have
 \begin{equation}\label{s4}
 \langle y+a \rangle^{s}=\exp_{p}(s\log_{p} \langle y+a \rangle)
 \end{equation} 
 for $(y,s)\in D$. Hence the two variable function
 $f(y,s)=\langle y+a\rangle^{s}$ is continuous on the domain
 $D$.
Since $D=\mathbb{Z}_{p}\times\mathbb{Z}_{p}$ is compact in the $p$-adic topology, for any fixed $a\in\mathbb{C}_{p}$ with $|a|_{p}>1$ it is bounded as a function for $(y,s)\in D$,
so there exists a positive constant $N_{a}$ such that for any $k\in\mathbb{N}$ and  $n\in\mathbb{N},
$\begin{equation}\label{l1+}\begin{aligned}\left|\frac{(-1)^{k}}{\langle k+a \rangle^{s+n-1}}\right|_{p}&=\left|(-1)^{k} \langle k+a \rangle^{1-s-n}\right|_{p}\\
&=\left|(-1)^{k} f(k,1-s-n)\right|_{p}\\&\leq N_{a} \end{aligned}\end{equation}
and by the non-archimedean property,  for any $r\in\mathbb{N}$,
\begin{equation}\label{l2+} \left|\sum_{k=1}^{p^{r}-1}\frac{(-1)^{k}}{\langle k+a \rangle^{s+n-1}}\right|_{p} \leq N_{a}.\end{equation}
Then combining (\ref{est1}) and (\ref{l2+}), for any fixed $a\in\mathbb{C}_{p}$ with $|a|_{p}>1$ and for any $r\in\mathbb{N}$ we have 
\begin{equation}\label{l3+}
\left|\frac{\prod_{j = 1}^{n - 1}(s-1+j)}{n!\omega_{v}^{n}(a)}\sum_{k=1}^{p^{r}-1}\frac{(-1)^{k}}{\langle k+a \rangle^{s+n-1}}\right|_{p}\leq N_{a}\cdot p^{nv_{p}(a)}\cdot n.
\end{equation} 
 Since $|a|_{p}>1$, i.e., $v_{p}(a)<0$, we have $\lim_{n\to\infty} N_{a}\cdot p^{nv_{p}(a)}\cdot n=0,$
which implies the series 
\begin{equation}N_{a} \label{l4+} \sum_{n=2}^{\infty}p^{nv_{p}(a)} n\end{equation}  is  convergent.
Finally by (\ref{l3+}), (\ref{l4+}) and the Weierstrass test (see \cite[p. 230, Theorem 5.1]{Lang2}), we see that the series 
$$\sum_{n=2}^\infty \frac{\prod_{j = 1}^{n - 1}(s-1+j)}{n!\omega_{v}^{n}(a)}\sum_{k=1}^{p^{r}-1}\frac{(-1)^{k}}{\langle k+a \rangle^{s+n-1}}$$ 
converges uniformly for  $r\in\mathbb{N}$. Then applying \cite[p. 185, Theorem 3.5]{Lang2} we conclude that  the limit $r\to\infty$  can be moved to the inside of the
above series, which is the desired result. \end{proof}

The following result ensures the convergence of (\ref{main1}), which is  a $p$-adic analogue of \cite[Lemma 2]{PK}.

\begin{lemma}\label{Lemma2}
The left hand side of (\ref{main1}) in Lemma \ref{Lemma1} converges  $p$-adically for $s\in\mathbb{Z}_{p}$ with $s\neq 1$ and $a\in\mathbb{C}_{p}$ with $|a|_{p}>1$.
\end{lemma}

\begin{proof}
By Proposition \ref{analytic} and Theorem \ref{analytic-zeta}, for $a\in\mathbb{C}_{p}$ with $|a|_{p}>1$, $\zeta_{p,E}(s,a)-\langle a \rangle^{1-s}$
is a $C^{\infty}$ function of $s$ on $\mathbb{Z}_{p}$. Since $\mathbb{Z}_{p}$ is compact in the $p$-adic topology, for any fixed $a\in\mathbb{C}_{p}$ with $|a|_{p}>1$ it is bounded as a function for $s\in\mathbb{Z}_{p}$,
i.e., there exists a positive constant $M_{a}$ such that
\begin{equation}\left|\label{l8} \zeta_{p,E}(s,a)-\langle a \rangle^{1-s}\right|_{p}\leq M_{a}. \end{equation}
Then combining (\ref{est1}) and (\ref{l8}),  for any fixed $a\in\mathbb{C}_{p}$ with $|a|_{p}>1$ we have
\begin{equation}
\left|\frac{\prod_{j = 1}^{n - 1}(s-1+j)}{n!\omega_{v}^{n}(a)}
\left(\zeta_{p,E}(s+n,a)-\langle a \rangle^{1-(s+n)}\right)\right|_{p}\leq M_{a}\cdot p^{nv_{p}(a)}\cdot n.
\end{equation}
Since $|a|_{p}>1$, i.e., $v_{p}(a)<0$, we have $\lim_{n\to\infty} M_{a}\cdot p^{nv_{p}(a)}\cdot n=0,$
which implies $$\lim_{n\to\infty}\left|\frac{\prod_{j = 1}^{n - 1}(s-1+j)}{n!\omega_{v}^{n}(a)}\left(\zeta_{p,E}(s+n,a)-\langle a \rangle^{1-(s+n)}\right)\right|_{p}=0,$$
thus the series $$\sum_{n=2}^\infty \frac{\prod_{j = 1}^{n - 1}(s-1+j)}{n!\omega_{v}^{n}(a)}\left(\zeta_{p,E}(s+n,a)-\langle a \rangle^{1-(s+n)}\right)$$
is convergent under the $p$-adic topology.
\end{proof}

The above result implies the following theorem.

\begin{theorem}\label{main theorem}
Let $T_{p}^{a}$  be as defined in (\ref{Tp}). Then $\zeta_{p,E}(s,a)$ formally satisfies the following differential equation
\begin{equation}\label{main2}
T_{p}^{a}\left[\zeta_{p,E}(s,a)-\langle a\rangle^{1-s}\right]=\frac{1}{s-1}\left(\langle a-1\rangle^{1-s}-\langle a\rangle^{1-s}\right)
\end{equation}
for $s\in\mathbb{Z}_{p}$ with $s\neq 1$ and $a\in\mathbb{C}_{p}$ with $|a|_{p}>1$.
\end{theorem}

\begin{proof}
Denote by $D_s^k := \frac{\partial ^k}{\partial s^k}$. For any analytic function $f(s)$ on $\mathbb{Z}_{p}$ and $n\in\mathbb{N}$ we have
\begin{equation}\label{Shift}
\begin{aligned}
\exp_{p}(nD)f(s)&= \left(id + \sum_{k=1}^\infty \frac{n^k}{k!} D_s^k\right)f(s)\\
&=f(s)+ \sum_{k=1}^\infty \frac{n^k}{k!}\frac{\partial ^k f(s)}{\partial s^k}\\
&=f(s+n),
\end{aligned}
\end{equation}
which maybe interpreted operationally through its formal Taylor expansion in $n.$
By Proposition \ref{analytic} and Theorem \ref{analytic-zeta}, for $a\in\mathbb{C}_{p}$ with $|a|_{p}>1$,
the function $\zeta_{p,E}(s,a)-\langle a \rangle^{1-s}$
is analytic for $s\in\mathbb{Z}_{p}.$
Thus from (\ref{Shift}) we get
\begin{equation}
\begin{aligned}
L_{p,n}^{a}\left[\zeta_{p,E}(s,a)-\langle a \rangle^{1-s}\right]&= P_{p,n}^{a}(s) \exp_{p}(nD)\left[\zeta_{p,E}(s,a)-\langle a \rangle^{1-s}\right]\\
&=P_{p,n}^{a}(s)\left(\zeta_{p,E}(s+n,a)-\langle a \rangle^{1-(s+n)}\right) \\
\end{aligned}
\end{equation}
for $n\geq 0$ and by the definition of $T_{p}^{a}$ (\ref{Tp}) and Lemma \ref{Lemma1}
\begin{equation}\label{3.17}
\begin{aligned}
&\quad T_{p}^{a}\left[\zeta_{p,E}(s,a)-\langle a \rangle^{1-s}\right]\\
&=\sum_{n=0}^\infty L_{p,n}^{a}\left[\zeta_{p,E}(s,a)-\langle a \rangle^{1-s}\right]\\
&=\sum_{n=0}^\infty P_{p,n}^{a}(s)\left(\zeta_{p,E}(s+n,a)-\langle a \rangle^{1-(s+n)}\right) \\
&=\frac{2}{s-1}\left(\zeta_{p,E}(s,a)-\langle a \rangle^{1-s}\right)
+\frac{1}{\omega_{v}(a)}\left(\zeta_{p,E}(s+1,a)-\langle a \rangle^{1-(s+1)}\right)\\
&\quad+\sum_{n=2}^\infty \frac{\prod_{j = 1}^{n - 1}(s-1+j)}{n!\omega_{v}^{n}(a)}\left(\zeta_{p,E}(s+n,a)-\langle a \rangle^{1-(s+n)}\right)\\
&=\frac{1}{s-1}\left(\langle a-1 \rangle^{1-s}-\langle a\rangle^{1-s}\right),
\end{aligned}
\end{equation}
which is the desired result.
\end{proof}

In what follows, we shall investigate the area of convergence for Theorem \ref{main theorem} and show that the operator $T_{p}^{a}$ applied to the $p$-adic Hurwitz-type Euler zeta function $\zeta_{p,E}(s,a)$ is convergent in certain area of the $p$-adic plane.
First we need to prove the following proposition.

\begin{proposition}\label{Taylorcon}
 Let $K$ be a finite extension of
 $\mathbb{Q}_{p}$ with ramification index $e$ over $\mathbb{Q}_{p}$  less than $p-1.$ Let $s\in\mathbb{C}_{p}$ with $|s|_{p}<r_{p}:=p^{\frac{1}{e}-\frac{1}{p-1}},$ and $a\in K\backslash\mathbb{Z}_{p}.$ For any $n \geq 2$ the series
$$\exp_{p}(nD)\left[\zeta_{p,E}(s,a) -\langle a \rangle^{1-s}\right]
=\sum_{k = 0}^\infty \frac{D_s^k \left( \zeta_{p,E}(s,a) -\langle a \rangle^{1-s}\right)}{k!}n^k $$
converges.
\end{proposition}

\begin{remark}\label{remark3}
This is mainly because the non-archimedean property of the $p$-adic metric and it is quite different from the complex situation for Hurwitz zeta functions. In that case, by \cite[Proposition 5]{PK}, we have  ``for any $s \in \mathbb{C}$, we can find some $N \geq 0$ so that the series
$$\exp(ND)\left[\zeta(s,a) - \frac{1}{a^s}\right]=\sum_{k = 0}^\infty \frac{D_s^k \left( \zeta(s,a) - \frac{1}{a^s}\right)}{k!}N^k $$ diverges."
\end{remark}

\begin{proof}[Proof of Propsoition \ref{Taylorcon}.]
Let $(\pi)$ be the maximal ideal of
 the ring of integers $O_{K}$ of $K$. Then $$|\pi|_{p}=|p|_{p}^{\frac{1}{e}}=\left(\frac{1}{p}\right)^{\frac{1}{e}}.$$ By Proposition \ref{analytic} and Theorem \ref{analytic-zeta}, given $a\in K\backslash\mathbb{Z}_{p},$
the function $\zeta_{p,E}(s,a) -\langle a \rangle^{1-s}$ is  analytic  for $$|s|_{p}<r_{p}:= |\pi|_{p}^{-1}p^{-1/(p-1)}=p^{\frac{1}{e}-\frac{1}{p-1}}.$$
 Fix $s_{0}\in\mathbb{C}_{p}$ with $|s_{0}|_{p}<r_{p}.$ For any $s\in\mathbb{C}_{p}$ with $|s-s_{0}|_{p}<r_{p}$, we have
$$|s|_{p}\leq \max\{|s-s_{0}|_{p},|s_{0}|_{p}\}<r_{p},$$ so the disc $\{s: |s-s_{0}|< r_{p}\}$ is contained in the disc  $\{s:|s|< r_{p}\}.$ In fact, $$\{s: |s-s_{0}|< r_{p}\}=\{s:|s|< r_{p}\}.$$
Thus $\zeta_{p,E}(s,a)$  can be expanded as a power series around $s_{0}$ with the radius of convergence equal to $r_{p}.$

Since $e < p-1$ as the assumption,  we have $r_{p}>1$ and for any $n\in\mathbb{N},$ we have $|(s_{0}+n)-s_{0}|_{p}=|n|_{p}\leq 1<r_{p}.$
From the discussion above, we have  the following convergent power series expansion of $\zeta_{p,E}(s,a)$ at $s_{0}$
$$\zeta_{p,E}(s_{0}+n,a) -\langle a \rangle^{1-(s_{0}+n)}
=\sum_{k = 0}^\infty \frac{D_s^k\big|_{s=s_{0}}  \left( \zeta_{p,E}(s,a) -\langle a \rangle^{1-s}\right)}{k!}n^k.$$
Then by the definition of exp$_{p}(nD)$ (\ref{Tp}), we see that
\begin{equation}\begin{aligned}
\zeta_{p,E}(s_{0}+n,a) -\langle a \rangle^{1-(s_{0}+n)}
&=\sum_{k = 0}^\infty \frac{D_s^k\big|_{s=s_{0}} \left( \zeta_{p,E}(s,a) -\langle a \rangle^{1-s}\right)}{k!}n^k\\
&=\exp_{p}(nD)\big|_{s=s_{0}} \left[\zeta_{p,E}(s,a) -\langle a \rangle^{1-s}\right],
\end{aligned}
\end{equation}
which is the desired result.
\end{proof}

From the above proposition we have the following result  which asserts  that the operator $T_{p}^{a}$ applied to the $p$-adic Hurwitz-type Euler zeta function $\zeta_{p,E}(s,a)$ is convergent
in the $p$-adic topology.

\begin{corollary}\label{Cor} Let $K$ be stated as in the Proposition \ref{Taylorcon}. Then
\begin{equation}
 T_{p}^{a}\left[\zeta_{p,E}(s,a)-\langle a \rangle^{1-s}\right]
=\sum_{n=1}^\infty P_{p,n}^{a}(s)\exp_{p}(nD)\left[\zeta_{p,E}(s,a)-\langle a \rangle^{1-s}\right]
\end{equation}
converges for $s\in\mathbb{Z}_{p}$ with $s\neq 1$ and $a\in K$ with $|a|_{p}>1.$
\end{corollary}

\begin{remark}\label{remark4}
Notice that in the complex situation, we have \begin{equation*}T\left[ \zeta (s, a) - \frac{1}{a^s} \right]=\sum_{n=0}^{\infty} p_n(s)\exp(nD)\left[\zeta(s,a)- \frac{1}{a^s}\right]\end{equation*} diverges for all complex numbers $s\in \mathbb{C}$
(see \cite[Theorem 8]{PK}).
\end{remark}

\begin{proof}[Proof of Corollary \ref{Cor}.]
By (\ref{3.17}) we have
\begin{equation}\label{3.20}
\begin{aligned}
&\quad T_{p}^{a}\left[\zeta_{p,E}(s,a)-\langle a \rangle^{1-s}\right]\\
&=\sum_{n=0}^\infty P_{p,n}^{a}(s)\exp_{p}(nD)\left[\zeta_{p,E}(s,a)-\langle a \rangle^{1-s}\right]\\
&=\frac{2}{s-1}\left(\zeta_{p,E}(s,a)-\langle a \rangle^{1-s}\right)
+\frac{1}{\omega_{v}(a)}\left(\zeta_{p,E}(s+1,a)-\langle a \rangle^{1-(s+1)}\right)\\
&\quad+\sum_{n=2}^\infty \frac{\prod_{j = 1}^{n - 1}(s-1+j)}{n!\omega_{v}^{n}(a)}\left(\zeta_{p,E}(s+n,a)-\langle a \rangle^{1-(s+n)}\right).
\end{aligned}
\end{equation}
Suppose that $s\in\mathbb{Z}_{p}$ with $s\neq 1$ and $a\in K$ with $|a|_{p}>1.$ By (\ref{Tp2}) and (\ref{est1}), for $n\geq 2$ we have 
$$|P_{p,n}^{a}(s)|_{p}= \left|\frac{\prod_{j = 1}^{n - 1}(s-1+j)}{n!\omega_{v}^{n}(a)}
\right|_{p}\leq p^{nv_{p}(a)}\cdot n
 $$
 and $\lim_{n\to\infty} P_{p,n}^{a}(s)=0.$ Then combining the conclusions of Proposition \ref{Taylorcon} and Lemma \ref{Lemma2}, for each $n \geq 2,$ both the series
$$\exp_{p}(nD)\left[\zeta_{p,E}(s,a) -\langle a \rangle^{1-s}\right]=\sum_{k = 0}^\infty \frac{D_s^k \left( \zeta_{p,E}(s,a) -\langle a \rangle^{1-s}\right)}{k!}n^k $$ and the right hand side of (\ref{3.20})
converge, which have established our result.
\end{proof}

\section*{Acknowledgement} The authors are enormously grateful to the anonymous referee for his/her very careful
reading of this paper, and for his/her many valuable and detailed suggestions. We  also thank Professor Lawrence C. Washington for pointing out a gap in the proof of Lemma \ref{Lemma1} of the original manuscript and for his helpful suggestions.

Su Hu is supported by the Natural Science Foundation of Guangdong Province, China (No. 2020A1515010170).  Min-Soo Kim is supported by the National Research Foundation of Korea(NRF) grant funded by the Korea government(MSIT) (No. 2019R1F1A1062499). 

\bibliography{central}

\end{document}